\documentclass[12pt,reqno,a4paper]{amsart}
\usepackage[latin1]{inputenc}
\usepackage[english]{babel}
\usepackage{latexsym}
\usepackage{amsmath}
\usepackage{amssymb}
\usepackage{enumitem}

\theoremstyle{plain}
\newtheorem{theorem}{Theorem}[section]

\newtheorem{lemma}[theorem]{Lemma}

\theoremstyle{definition}

\theoremstyle{remark}
\newtheorem{remark} [theorem]{Remark}

\newcommand{\DD}[2][]{\frac{\mathrm{d}^2{#1}}{\mathrm{d}{#2}^2}}

\newcommand{\norm}[2][]{\left\|#2\right\|_{#1}}

\newcommand{\R}{\mathbb{R}}

\newcommand{\e}{\mathrm{e}}

\newcommand{\dd}{\, \mathrm{d}}

\begin{document}
\title[Improved spectral inequalities on the semi-axis]{Improved sharp spectral inequalities for Schr{\"o}dinger operators on the semi-axis}

\author[L. Schimmer]{Lukas Schimmer}
\address{Lukas Schimmer, Institut Mittag--Leffler, The Royal Swedish Academy of Sciences, Djursholm 182 60, Sweden, Current address: Edinburgh, UK}
\email{lukas.schimmer@gmail.com}

\maketitle

\begin{abstract}
We prove a Lieb--Thirring inequality for Schr{\"o}dinger operators $-\DD{x}+V$ on the semi-axis with Robin boundary condition at the origin. The result improves on a bound obtained by P.~Exner, A.~Laptev and M.~Usman [Commun.~Math.~Phys. \textbf{362}(2), 531--541 (2014)] albeit under the additional assumption $V\in L^1(\R_+)$. The main difference in our proof is that we use the double commutation method in place of the single commutation method. We also establish an improved inequality in the case of a Dirichlet boundary condition.
\end{abstract}

\section{Introduction}

In their proof of stability of matter, Lieb and Thirring \cite{Lieb1975,Lieb1976} introduced the bound
\begin{align*}
\sum_{j\ge 1} \vert \lambda_j\vert^\gamma\le L_{\gamma,d}\int_{\R^d}V(x)_-^{\gamma+\frac d2}\dd x
\end{align*}
for the negative eigenvalues $\lambda_1\le\lambda_2\le\dots\le0$ of a Schr{\"o}dinger operator $-\Delta+V$ on $L^2(\R^d)$ with real-valued potential $V$ that decays sufficiently fast. Here and below $a_-=(\vert a\vert -a)/2$ denotes the negative part of a real variable $a\in\R$. The bound was proved for any $\gamma>\max(0,1-\frac d2)$ and was later extended to the endpoint cases $d=1,\gamma=\frac12$ and $d=3,\gamma=0$ in \cite{Weidl1996} and \cite{Cwikel1977,Lieb1976b,Rozenblum1976}, respectively. The sharp constants $L_{\gamma,d}$, which importantly do not depend on $V$, have been subject of intense investigation over the last 45 years \cite{Schimmer2022LT}. 

The case $d=1$ and $\gamma=\frac32$ has proved especially accessible to mathematical investigations due to its connection to trace formulae. The sharp constant $L_{\frac32,1}=\frac{3}{16}$ was established even before Lieb and Thirring's original papers by Gardner, Greene, Kruskal and Miura \cite{Gardner1961}. The authors  considered the Buslaev--Faddeev--Zaharov trace formula \cite{Buslaev1960, Zaharov1971}
\begin{align*}
\sum_{j\ge1} \vert \lambda_j\vert ^{\frac32}+\frac{3}{\pi}\int_{\R_+}k^2 \log \vert a(k)\vert \dd k=\frac{3}{16}\int_{\R} V(x)^2\dd x
\end{align*}
for the negative eigenvalues $\lambda_j$ of $-\DD{x}+V$ on $L^2(\R)$ and noted that the scattering coefficient satisfies $\vert a(k)\vert \ge 1$. This yields the sharp inequality
\begin{align}
\sum_{j\ge 1}\vert \lambda_j\vert ^\frac32
\le\frac{3}{16}\int_\R V(x)^2\dd x\,. 
\label{eq:LTR} 
\end{align}  
 An extension of \eqref{eq:LTR} to matrix-valued potentials by Laptev and Weidl \cite{Laptev2000} was crucial in establishing the sharp Lieb--Thirring constants $L_{\gamma,d}$ for $\gamma\ge3/2$ in all dimensions $d\ge1$. Note that  the trace formula also yields a bound on the integral involving the scattering coefficient, which has proved very useful in the investigation of the absolute continuity of the spectrum of the Schr{\"o}dinger operator \cite{Deift1999}.

In this short note, we consider the Schr{\"o}dinger operator
\begin{align*}
H=-\DD{x}+V(x)
\end{align*}
on $L^2(\R_+)$ with real-valued potential and Robin boundary condition
\begin{align*}
\varphi'(0)-\sigma_0\varphi(0)=0
\end{align*}
where $\sigma_0\in\R$. If the potential $V$ is sufficiently smooth and decays sufficiently fast, the negative spectrum of $H$ consists of discrete eigenvalues $\lambda_1\le\lambda_2\le\dots\le0$ with corresponding eigenfunctions $\varphi_j$. While trace formulae have also been established in this setting \cite{Demirel2011}, there is no known analogue of $\vert a(k)\vert \ge 1$. Thus Lieb--Thirring inequalities have to be proved by different means and could in turn be used to shed more light on the scattering coefficient.  
Our main result is the following Lieb--Thirring type bound.
\begin{theorem}\label{th:main}
Let $V\in L^1(\R_+)\cap L^2(\R_+)$.  The negative eigenvalues $\lambda_j$ of $-\DD{x}+V$ with Robin boundary condition $\varphi'(0)-\sigma_0\varphi(0)=0$ satisfy
\begin{align*}
\sum_{j\ge1} \vert \lambda_j\vert ^\frac32+\frac14\sum_{j\ge1}(\sigma_{j}^3-\sigma_{j-1}^3)\le \frac{3}{16}\int_0^\infty V(x)^2\dd x +\frac34\sum_{j\ge1}\vert \lambda_j\vert (\sigma_{j}-\sigma_{j-1})
\end{align*}
where
\begin{align*}
\sigma_j=\sigma_{j-1}+\frac{\vert \varphi_j(0)\vert ^2}{\norm{\varphi_j}^2}\,,\quad j=1,2,\dots
\end{align*}
and $\varphi_j$ denotes the eigenfunction to $\lambda_j$.
\end{theorem}
\begin{remark}
From the proof it is clear that the bound also holds if each of the three sums only extends to $j\le N$ for some cutoff $N\ge1$ (with additional terms replaced by 0 if there are fewer than $N$ negative eigenvalues).
All four quantities in the inequality above are then non-negative and non-decreasing in $N$. Thus the two sides of the inequality are also well-defined in the case of infinitely many negative eigenvalues, though the theorem does not make any assertion about the finiteness of the two series involving $\sigma_j$. However, the difference of the fourth and second term is always bounded from above. Some explicit upper bounds that could be useful in applications will be discussed in Section \ref{sec:comp}. 
Note that finiteness of the discrete spectrum holds for example if $\int_{0}^\infty(1+x)\vert V(x)\vert \dd x<\infty$ and in particular if $V\in\mathcal{C}_0^\infty([0,\infty))$.
\end{remark}

In the special case of a Dirichlet boundary condition, we obtain the following.
\begin{theorem}\label{th:Dir}
Let $V\in L^1(\R_+)\cap L^2(\R_+)$.  The negative eigenvalues $\lambda_j$ of $-\DD{x}+V$ with Dirichlet boundary condition $\varphi(0)=0$ satisfy
\begin{align*}
\sum_{j\ge1}\vert \lambda_j\vert ^\frac32\le \frac{3}{16}\int_0^\infty V(x)^2\dd x -\frac34\sum_{j\ge1}\frac{\vert \varphi_j'(0)\vert ^2}{\norm{\varphi_j}^2}
\end{align*}
where $\varphi_j$ denotes the eigenfunction to $\lambda_j$.
\end{theorem}

\begin{remark}

From the proof it is again clear that the bound also holds if each of the two sums only extends to $j\le N$ for some cutoff $N\ge1$. Both sums are non-negative and non-decreasing in $N$. Letting $N\to\infty $ we can conclude that under the assumptions of the theorem the two series are both finite, even in the case of infinitely many eigenvalues. 
\end{remark}
 
Note that the inequality of Theorem \ref{th:Dir} without the negative last term can be obtained from the whole line result \eqref{eq:LTR}.  The inequality of Theorem \ref{th:main} should be compared to the following result by Exner, Laptev and Usman \cite{Exner2014} which was established in the same setting but without the assumption $V\in L^1(\R_+)$.

\begin{theorem}[{\cite[Theorem 1.1]{Exner2014}}]\label{th:LTELU}
Let $V\in L^2(\R_+), V\le0$. The negative eigenvalues $\lambda_j$ of $-\DD{x}+V$ with Robin boundary condition $\varphi'(0)-\sigma_0\varphi(0)=0$ satisfy
\begin{align*}
\frac12\vert \lambda_1\vert ^\frac32+\sum_{j\ge 2}\vert \lambda_j\vert ^\frac32
\le\frac{3}{16}\int_0^\infty V(x)^2\dd x-\frac34\vert \lambda_1\vert \sigma_0+\frac14\sigma_0^3\,.
\end{align*} 
\end{theorem}

Theorem \ref{th:LTELU} shows that compared to the whole line case \eqref{eq:LTR}, the boundary condition at zero leads to a change in the term corresponding to $\lambda_1$ in the Lieb--Thirring bound. Our result in Theorem \ref{th:main} aims to further elaborate on the influence of the boundary condition. In Section \ref{sec:comp} we will show that the additional terms in Theorem \ref{th:main} strengthen the inequality. In particular, Theorem \ref{th:LTELU} can be obtained from our result. While the inequality in Theorem \ref{th:main} may be difficult to use in applications due to the necessary knowledge of $\sigma_j$ (and thus of $\vert \varphi_j(0)\vert /\norm{\varphi_j}$) for $j\ge1$, we will show in Section \ref{sec:comp} how in some cases the bound can be weakened to a form that does not require this information. Some of these results cannot be obtained directly from Theorem \ref{th:LTELU}. Before we prove the main result, it is worth pointing out the differences in our proof method compared to the existing literature. 

For $d=1$ the so-called commutation method has proved valuable in establishing sharp Lieb--Thirring inequalities. This method goes back to the idea of inserting eigenvalues into the spectrum of differential operators and was first discussed by Jacobi \cite{Jacobi1837}, Darboux \cite{Darboux1882} and Crum \cite{Crum1955}. A rigorous characterisation can be found in \cite{Deift1978, Gesztesy1993, Gesztesy1996}. For the purpose of proving Lieb--Thirring inequalities, the method is reversed and eigenvalues are successively removed from the spectrum, starting with the lowest, $\lambda_1$. To this end one constructs a first-order differential operator $D$ that factorises the original Schr{\"o}dinger operator as $-\DD{x}+V=DD^*+\lambda_1$. Commuting $D$ and $D^*$ leads to a new operator $-\DD{x}+V_1=D^*D+\lambda_1$, which has the same spectrum as the original operator with the exception of the eigenvalue $\lambda_1$.  In order to obtain a spectral inequality, it is necessary to establish a connection between integrals of powers of the potentials $V$ and $V_1$ (such as  $\int V^2\dd x$), and the eigenvalue $\lambda_1$. Assuming that there are only finitely many negative eigenvalues $\lambda_1,\dots,\lambda_N$, repetition of this process removes all of these eigenvalues from the spectrum and one eventually obtains an identity that links $\lambda_1,\dots,\lambda_N$ to integrals of $V$ and some potential $V_N$ that corresponds to a Schr{\"o}dinger operator without negative eigenvalues. If this last term has a definite sign, an inequality can be obtained. 

In the case of a Schr{\"o}dinger operator on the real line, the commutation method was first used by Schmincke \cite{Schmincke1978} to prove the lower bound 
\begin{align}
\sum_{j\ge1}\vert \lambda_j\vert ^\frac12\ge-\frac14\int_{\R} V(x)\dd x\,.
\label{eq:SchminckeR}
\end{align}
Subsequently, it has been applied to provide a new, direct proof of \eqref{eq:LTR} in the case of matrix-valued potentials  \cite{Benguria2000} (as first established by Laptev and Weidl \cite{Laptev2000}) and to prove similar inequalities for fourth-order differential operators \cite{Hoppe2006} and Jacobi operators \cite{Schimmer2015}. In a slight variation, this proof method has also been used to establish Theorem \ref{th:LTELU}. Here, after removing the first eigenvalue, one obtains a Schr{\"o}dinger operator with Dirichlet boundary condition at zero. The Lieb--Thirring inequality is then proved by  continuing the problem to the whole line and applying \eqref{eq:LTR}. Our Theorem \ref{th:Dir} shows that such an approach cannot yield a sharp inequality if the potential supports more than one eigenvalue (under the additional condition $V\in L^1(\R_+)$). Recently the same variation of the commutation method has been applied to fourth-order operators on the semi-axis \cite{Zia2019}. 

In all of theses results, the applied method is more precisely known as the single commutation method. In comparison, the so-called double commutation method  \cite{Gesztesy1993,Gesztesy1996} involves an additional step where after commuting $D,D^*$  the resulting operator is again factorised using a new first-order operator $D_\gamma$ such that $-\DD{x}+V_1=D^*D+\lambda_1=D_\gamma^* D_\gamma+\lambda_1$. Applying a second commutation, one obtains yet another Schr{\"o}dinger operator $-\DD{x}+V_{\gamma,1}=D_\gamma D_\gamma^*+\lambda_1$ that has the same spectrum as the original operator with the exception of the eigenvalue $\lambda_1$. This method has several advantages compared to the single commutation method. For example, it allows to remove eigenvalues in arbitrary order, as it does not require the corresponding eigenfunction to have no zeros. In our case, its main advantage is that after the first step, we do not obtain a Schr{\"o}dinger operator with Dirichlet boundary condition, but rather one with a new Robin boundary condition. This leads to the additional terms in Theorem \ref{th:main} compared to Theorem \ref{th:LTELU}. To the best of our knowledge, the double commutation method has not been used previously in the context of Lieb--Thirring inequalities. In \cite{Boumenir2009} the closely related Gelfand--Levitan method \cite{Levitan1987} was applied in the same setting as in this note to obtain the lower bound
\begin{align*}
\sum_{j\ge1}\vert \lambda_j\vert ^\frac12\ge
-\frac14\int_{\R} V(x)\dd x-\frac{1}{4}\sigma_0+\frac14\sum_{j\ge1}\frac{\vert \varphi_j(0)\vert ^2}{\norm{\varphi_j}^2}
\end{align*}
for the operator $-\DD{x}+V$ on $L^2(\R_+)$ with Robin boundary condition. This result shows that the boundary condition at the origin influences Schmincke's inequality \eqref{eq:SchminckeR} in a similar way as it influences the Lieb--Thirring inequality \eqref{eq:LTR} in Theorem \ref{th:main}.

In Section \ref{sec:double} we will introduce the double commutation method in more detail and subsequently we will use it in Section \ref{sec:proof} to prove Theorem \ref{th:main} and Theorem \ref{th:Dir}.

\section{The double commutation method}\label{sec:double}
For brevity we restrict ourselves to the case at hand, i.e.~a Schr{\"o}dinger operator $H=-\DD{x}+V$ on $L^2(\R_+)$ with Robin boundary condition $\varphi'(0)-\sigma\varphi(0)=0$. For comparison we first state the single commutation method, details of which can be found in \cite{Deift1978}.
\begin{theorem}
Let $\varphi$ be an eigenfunction of $H=-\DD{x}+V$ to the lowest eigenvalue $\lambda$. Then the operator $H_\lambda=-\DD{x}+V_\lambda$ with potential
\begin{align*}
V_\lambda(x)=V(x)-2\DD{x}\log\varphi(x)
\end{align*}
and with Dirichlet boundary condition
\begin{align*}
\varphi(0)=0
\end{align*}
has spectrum $\sigma(H_\lambda)=\sigma(H)\setminus\{\lambda\}$.
\end{theorem}

\begin{remark}
As discussed in the introduction, the result is the consequence of the factorisation $H=DD^*+\lambda$ and  $H_\lambda=D^*D+\lambda$, where more precisely $D=\frac{\mathrm{d}}{\mathrm{d}x}+\frac{\varphi'}{\varphi}$. 
\end{remark}
 
The spectral characterisation of the double commutation method was first achieved in \cite{Gesztesy1993} for Schr{\"o}dinger operators on $L^2(\R)$ as well as on $L^2(\R_+)$ with Dirichlet boundary condition at the origin. The results were extended to Sturm--Liouville operators on arbitrary intervals with Robin boundary conditions in \cite{Gesztesy1996}, from where we take the following result \cite[Theorem 3.2]{Gesztesy1996} (see also \cite[Remark 3.3 (i)]{Gesztesy1996}). 
\begin{theorem}\label{th:double}
Let $\varphi$ be an eigenfunction of $H=-\DD{x}+V$ with eigenvalue $\lambda$ and let  
$\gamma=-1/\norm{\varphi}^2$. Then the operator $H_\lambda=-\DD{x}+V_\lambda$ with potential
\begin{align*}
V_\lambda(x)=V(x)-2\DD{x}\log\!\left(1+\gamma\int_0^x\vert \varphi(t)\vert ^2\dd t\right)
\end{align*}
and with Robin boundary condition
\begin{align*}
\psi'(0)-\sigma_\lambda\psi(0)=0\,,\qquad \sigma_\lambda=\sigma+\frac{\vert \varphi(0)\vert ^2}{\norm{\varphi}^2}
\end{align*}
has point spectrum $\sigma_p(H_\lambda)=\sigma_p(H)\setminus\{\lambda\}$. Furthermore, $\psi$ is an eigenfunction of $H$ with eigenvalue $\eta\neq\lambda$ if and only if
\begin{align*}
\psi_\lambda(x)=\psi(x)-\gamma\widetilde{\varphi}(x)\int_0^x\psi(t)\overline{\varphi(t)}\dd t
\end{align*}
is an eigenfunction of $H_\lambda$ with eigenvalue $\eta\neq\lambda$ where the function $\widetilde{\varphi}$ is defined as
\begin{align*}
\widetilde{\varphi}(x)=\frac{\varphi(x)}{1+\gamma\int_0^x\vert \varphi(t)\vert ^2\dd t}\,.
\end{align*} 

\end{theorem}

\begin{remark}\label{rem:double}
In the notation of \cite{Gesztesy1996}, the boundary condition of $H_\lambda$ is given by the vanishing Wronskian $\psi(0)\varphi'(0)-\psi'(0)\varphi(0)=0$, which can easily be reduced to the one given above. 
As mentioned in the introduction, the double commutation method relies on a second factorisation $D^*D+\lambda=D_\gamma^*D_\gamma+\lambda$, where more precisely $D_\gamma=\frac{\mathrm{d}}{\mathrm{d}x}+\frac{\widetilde{\varphi}'}{\widetilde{\varphi}}$. 
\end{remark}

\section{The proofs of Theorem \ref{th:main} and Theorem \ref{th:Dir}}\label{sec:proof}
In many cases, proofs of Lieb--Thirring inequalities initially restrict to compactly supported potential $V$ and then use an approximation argument to extend the result to more general $V\in L^{\gamma+d/2}(\R^d)$. Since the bound in Theorem \ref{th:main} contains the terms $\sigma_j$, in our case such an approximation argument would necessarily have to establish the continuous dependence of the eigenfunctions on the potential in terms of the norm on $L^2(\R_+)$. To avoid this argument altogether, our proof will not restrict to compactly supported potentials. Establishing the required asymptotic behaviour of eigenfunctions is then more technical and relies on the additional assumption $V\in L^1(\R_+)$. This assumption is also necessary in the proof of the corresponding trace formula \cite{Demirel2011}. We do not know whether Theorem \ref{th:main} holds true without it. 

\subsection{The proof of Theorem \ref{th:main}} 
Let $\varphi_1$ now be the eigenfunction for the eigenvalue $\lambda_1$ and let $\gamma_1=-1/\norm{\varphi_1}^2$. As a ground state, $\varphi_1$ does not vanish anywhere (see e.g.~\cite{Exner2014} for a proof in this setting). It can thus be chosen to be strictly positive. Note that the behaviour of $\varphi_1$ at the origin is characterised by the boundary condition
\begin{align}
\varphi_1'(0)-\sigma_0\varphi_1(0)=0\,.
\label{eq:phi0}
\end{align}
For large $x$ the asymptotic behaviour 
\begin{align}
\lim_{x\to\infty}\varphi_1(x)\e^{\sqrt{\vert \lambda_1\vert }x}=C_1\,,\quad
\lim_{x\to\infty}\varphi_1'(x)\e^{\sqrt{\vert \lambda_1\vert }x}=-C_1\sqrt{\vert \lambda_1\vert }
\label{eq:phix0}
\end{align}
holds with some $C_1>0$. This is a consequence of the additional assumption $V\in L^1(\R_+)$ (see e.g.~\cite[Lemma 1]{Boumenir2009} which uses \cite[Theorem 8, Section 22]{Naimark1968}).

By Theorem \ref{th:double} the operator $H_1=-\DD{x}+V_1$ with potential
\begin{align*}
V_1(x)=V(x)-2\DD{x}\log\!\left(1+\gamma_1\int_0^x\vert \varphi_1(t)\vert ^2\dd t\right)
\end{align*}
and Robin boundary condition
\begin{align*}
\varphi'(0)-\sigma_1\varphi(0)=0\,,\qquad \sigma_1=\sigma_0+\frac{\vert \varphi_1(0)\vert ^2}{\norm{\varphi_1}^2}
\end{align*}
has only the negative eigenvalues $\lambda_2\le\lambda_3\le\dots\le0$. The potential can be written as $V_1=V-2G'$ with
\begin{align*}
G(x)=\frac{\gamma_1\varphi_1(x)^2}{1+\gamma_1\int_0^x\vert \varphi_1(t)\vert ^2\dd t}
\end{align*}
which can be further decomposed into $G=F-\widetilde{F}$ with
\begin{align*}
F(x)=\frac{\varphi_1'(x)}{\varphi_1(x)}\,,\qquad \widetilde{F}(x)=\frac{\widetilde{\varphi}_1'(x)}{\widetilde{\varphi}_1(x)}
\end{align*}
and 
\begin{align*}
\widetilde{\varphi}_1(x)=\frac{\varphi_1(x)}{1+\gamma_1\int_0^x\vert \varphi_1(t)\vert ^2\dd t}\,.
\end{align*}

\begin{lemma}\label{lem:F}
The functions $F$ and $\widetilde{F}$ solve the first-order differential equations 
\begin{align*}
F^2+F'=V-\lambda_1\,,\qquad \widetilde{F}^2-\widetilde{F}'+2 F'=V-\lambda_1
\end{align*}
with boundary conditions
\begin{align*}
F(0)&=\sigma_0\,,&\widetilde{F}(0)&=\sigma_1\,,\\
\lim_{x\to\infty} F(x)&=-\sqrt{\vert \lambda_1\vert }\,,&\lim_{x\to\infty}\widetilde{F}(x)&=\sqrt{\vert \lambda_1\vert }\,.
\end{align*}
\end{lemma}

\begin{proof}
The differential equation for $F$ can be found in several applications of the single commutation method. It is an immediate consequence of the eigenequation for $\varphi_1$
\begin{align*}
F(x)^2+F'(x)=\frac{\varphi_1'(x)^2+\varphi_1''(x)-\varphi_1'(x)^2}{\varphi_1(x)^2}=V(x)-\lambda_1\,.
\end{align*}
The boundary conditions follow from \eqref{eq:phi0} and \eqref{eq:phix0}. 
For $\widetilde{F}$ we compute that
\begin{align*}
\widetilde{F}(x)^2-\widetilde{F}'(x)=F(x)^2-F'(x)+G'(x)-2F(x)G(x)+G(x)^2
\end{align*}
and the differential equation can be proved by verifying that $G'(x)-2F(x)G(x)+G(x)^2=0$. 
The boundary condition at the origin is a consequence of \eqref{eq:phi0} while for $x\to\infty$ we use \eqref{eq:phix0} and l'Hospital's rule to compute
\begin{align*}
\lim_{x\to\infty}\widetilde{F}(x)&=\lim_{x\to\infty}\left(\frac{\varphi_1'(x)}{\varphi_1(x)}-\frac{\gamma_1\vert \varphi_1(x)\vert ^2}{1+\gamma_1\int_0^x\vert \varphi_1(t)\vert ^2\dd t}\right)\\
&
=-\sqrt{\vert \lambda_1\vert }-\lim_{x\to\infty}\frac{2\varphi_1(x)\varphi_1'(x)}{\varphi_1(x)^2}
=\sqrt{\vert \lambda_1\vert }\,.
\qedhere
\end{align*}
\end{proof}

We first note that
\begin{align*}
\int_0^\infty V_1(x)^2\dd x=\int_0^\infty V(x)^2\dd x+4 \int_0^\infty G'(x)\big(G'(x)-V(x)\big)\dd x\,.
\end{align*}
The last term on the right-hand side can be computed explicitly by using Lemma \ref{lem:F} 
\begin{align*}
&\int_0^\infty G'(x)\big(G'(x)-V(x)\big)\dd x\\
&=\int_0^\infty F'(x)\big(F'(x)-V(x)\big)\dd x
-\int_0^\infty \widetilde{F}'(x)\big(2F'(x)-\widetilde{F}'(x)-V(x)\big)\dd x\\
&=-\int_0^\infty F'(x)\big(\lambda_1+F(x)^2\big)\dd x
+\int_0^\infty \widetilde{F}'(x)\big(\lambda_1+\widetilde{F}(x)^2\big)\dd x\\
&=\Big[\vert \lambda_1\vert  F(x)-\frac13F(x)^3-\vert \lambda_1\vert  \widetilde{F}(x)+\frac13 \widetilde{F}(x)^3\Big]_{x=0}^{x=\infty}\\
&=-\frac{4}{3}\vert \lambda_1\vert ^\frac32+\vert \lambda_1\vert (\sigma_1-\sigma_0)-\frac13(\sigma_1^3-\sigma_0^3)\,.
\end{align*}
Thus we arrive at
\begin{align*}
\int_0^\infty V_1(x)^2\dd x
=-\frac{16}{3}\vert \lambda_1\vert ^\frac32+4\vert \lambda_1\vert (\sigma_1-\sigma_0)-\frac43(\sigma_1^3-\sigma_0^3)+\int_0^\infty V(x)^2\dd x\,.
\end{align*}

We aim to repeat the process and thus check whether $V_1$ satisfies the assumptions of Theorem \ref{th:main}. The identity above shows that $V_1\in L^2(\R_+)$. In  \cite[Lemma 2]{Boumenir2009} it is stated that $V_1\in L^1(\R_+)$, arguing that  $\vert G'\vert \in L^1(\R_+)$ since $G'(x)\ge 0$ for sufficiently large $x$. The latter is claimed to be a consequence of the asymptotics of $\varphi_1$. Unfortunately we could not fill in all of the details of the argument. In particular we could not rule out that $G'$ oscillates as $x\to\infty$. We instead present an argument that avoids investigating the integrability of $V_1$ altogether. In the computations above, the property $V\in L^1(\R_+)$ was only used to prove the asymptotic behaviour of the ground state $\varphi_1$ of $H$. More generally, the condition $V\in L^1(\R_+)$ guarantees that the eigenfunctions $\varphi_j$ of $H$ satisfy
\begin{align*}
\lim_{x\to\infty}\varphi_j(x)\e^{\sqrt{\vert \lambda_j\vert }x}=C_j\,,\quad
\lim_{x\to\infty}\varphi_j'(x)\e^{\sqrt{\vert \lambda_j\vert }x}=-C_j\sqrt{\vert \lambda_j\vert }
\end{align*} 
with $C_j\neq0$. These results already imply similar asymptotics for the eigenfunctions $\psi_j$ of $H_1$ without the need to establish $V_1\in L^1(\R_+)$. To this end we note that by Theorem \ref{th:double}
\begin{align}
\psi_j(x)
=\varphi_{j+1}(x)+\gamma_1\widetilde{\varphi}_1(x)\int_x^\infty\varphi_{j+1}(t)\overline{\varphi_1(t)}\dd t\,.
\label{eq:newef}
\end{align}
Using l'Hospital's rule it is straightforward to compute the three limits 
\begin{align*}
\lim_{x\to\infty}\widetilde{\varphi}_1(x)\e^{-\sqrt{\vert \lambda_1\vert }x}
&=-\frac{2}{C_1\gamma_1}\sqrt{\vert \lambda_1\vert }\,,\\
\lim_{x\to\infty}\widetilde{\varphi}_1'(x)\e^{-\sqrt{\vert \lambda_1\vert }x}
&=-\frac{2}{C_1\gamma_1}\vert \lambda_1\vert\,, \\
\lim_{x\to\infty}\int_x^\infty\varphi_{j+1}(t)\overline{\varphi_1(t)}\dd t
\,\e^{\sqrt{\vert \lambda_1\vert x}}\e^{\sqrt{\vert \lambda_{j+1}\vert x}}
&=\frac{C_1C_{j+1}}{\sqrt{\vert \lambda_1\vert }+\sqrt{\vert \lambda_{j+1}\vert }}\,.
\end{align*}
From \eqref{eq:newef} we then obtain the desired asymptotics
\begin{align*}
\lim_{x\to\infty}\psi_j(x)\e^{\sqrt{\vert \lambda_j\vert }x}=D_j\,,\quad
\lim_{x\to\infty}\psi_j'(x)\e^{\sqrt{\vert \lambda_j\vert }x}=-D_j\sqrt{\vert \lambda_j\vert }
\end{align*}
with $D_j=C_{j+1}(\sqrt{|\lambda_{j+1}|}-\sqrt{|\lambda_{1}|})/(\sqrt{|\lambda_{j+1}|}+\sqrt{|\lambda_{1}|})\neq0$. 

We can thus repeat the process for $H_1$ and remove $\lambda_2$ from its spectrum. While the eigenfunctions of $H_1$ are different to those of $H$, the relevant quantities in the definition of $\sigma_2$ importantly do not differ. More precisely, \eqref{eq:newef} allows us to conclude that $\psi_1(0)=\varphi_2(0)$ and furthermore that $\norm{\psi_1}^2=\norm{\varphi_2}^2$, as shown in \cite[Lemma 2.1]{Gesztesy1996}. Thus $\sigma_2$ can be written as $\sigma_2=\sigma_1+\vert {\psi}_1(0)\vert ^2/\norm{{\psi}_1}^2=\sigma_1+\vert {\varphi}_2(0)\vert ^2/\norm{{\varphi}_2}^2$. 

We can continue in this manner, noting that in each application of the double commutation method, the desired eigenfunction asymptotics inductively hold true. This yields the identity
\begin{align*}
&\int_0^\infty V_N(x)^2\dd x\\
&=-\frac{16}{3}\sum_{j=1}^N\vert \lambda_j\vert ^\frac32+4\sum_{j=1}^N\vert \lambda_j\vert (\sigma_j-\sigma_{j-1})-\frac43(\sigma_N^3-\sigma_0^3)+\int_0^\infty V(x)^2\dd x
\end{align*}
after $N$ steps. Since the left-hand side is non-negative we obtain the inequality
\begin{align*}
\sum_{j=1}^N\vert \lambda_j\vert ^\frac32+\frac14\sum_{j=1}^N(\sigma_{j}^3-\sigma_{j-1}^3)\le \frac{3}{16}\int_0^\infty V(x)^2\dd x +\frac34\sum_{j=1}^N\vert \lambda_j\vert (\sigma_{j}-\sigma_{j-1})\,.
\end{align*}
If the number of negative eigenvalues is finite, this is  already the desired bound. In the case of infinitely many eigenvalues, we can let $N\to\infty$ as all four terms are positive and non-decreasing in $N$.

\subsection{The proof of Theorem \ref{th:Dir}}
We start with the following observation. 
\begin{remark}\label{rem:zero}
We recall that $F$ and $\widetilde{F}$ in Lemma \ref{lem:F} were well-defined, since under the assumptions of Theorem \ref{th:main} the ground state $\varphi_1$ does not have any zeros. This fact was subsequently also used in the proof of the lemma. Note, however, that the decomposition $G=F-\widetilde{F}$ was only necessary in order to evoke similarities to the single commutation method and to simplify the computations. It can also be checked directly that the identity
\begin{align*}
&G'(x)\big(G'(x)-V(x)\big)=\\
&\phantom{=}\frac{\mathrm{d}}{\mathrm{d}x}\left(\vert \lambda_1\vert G(x)
-\frac{\gamma_1\varphi_1'(x)^2\Phi_1(x)^2-\gamma_1^2\varphi_1'(x)\varphi_1(x)^3\Phi_1(x)+\frac13\gamma_1^3\varphi_1(x)^6}{\Phi_1(x)^3}\right)
\end{align*}
holds, where $\Phi_1(x)=1+\gamma_1\int_0^x\vert \varphi_1(t)\vert ^2\dd t$. Here, all involved quantities are well-defined even if $\varphi_1$ has zeros. 
This shows that the double commutation method does not require us to remove the eigenvalues in increasing order. Furthermore, in a more general setting, the double commutation method could be used to remove eigenvalues in gaps of the essential spectrum other than the lowest one. 
\end{remark}
The above remark shows that we can apply the double commutation method to the Schr{\"o}dinger operator $-\DD{x}+V$ on $L^2(\R_+)$ with Dirichlet boundary condition at the origin. After the initial step, the operator $H_1=-\DD{x}+V_1$ is characterised (see Remark \ref{rem:double}) by the vanishing Wronskian $\psi(0)\varphi_1'(0)-\psi'(0)\varphi_1(0)=0$ which reduces to $\psi(0)=0$. Following the procedure above, we obtain the identity
\begin{align*}
\int_0^\infty V_1(x)^2\dd x=-\frac{16}{3}\vert \lambda_1\vert ^\frac32-4\frac{\vert \varphi_1'(0)\vert ^2}{\norm{\varphi_1}^2}+\int_0^\infty V(x)^2\dd x\,.
\end{align*}
From \eqref{eq:newef} we see that $\psi_1'(0)=\varphi_2'(0)$. We can then continue removing eigenvalues from the spectrum.  Repeating the process for altogether $N$ steps and using again that $\int_0^\infty V_N(x)^2\dd x\ge0$  we obtain
\begin{align*}
\sum_{j=1}^N\vert \lambda_1\vert ^\frac32+\frac{3}{4}\sum_{j=1}^N\frac{\vert \varphi_j'(0)\vert ^2}{\norm{\varphi_j}^2}\le\frac{3}{16}\int_0^\infty V(x)^2\dd x\,.
\end{align*}
This finishes the proof if the operator has only finitely many eigenvalues. The general case follows from taking $N\to\infty$ and noting that all three terms are non-negative and non-decreasing in $N$.

\section{Comparison and simplifications}\label{sec:comp}

\subsection{Comparison to Theorem \ref{th:LTELU}}
Under the assumptions of Theorem \ref{th:main} and if $V\le 0$,  the presented inequality is stronger than the result of Theorem \ref{th:LTELU}. To this end we note that by definition $\sigma_{j}-\sigma_{j-1}\ge0$ as well as $\vert\lambda_j\vert\le\vert\lambda_1\vert$ and thus for any $N\ge 1$
\begin{align}
\frac34\sum_{j=1}^N\vert \lambda_j\vert (\sigma_{j}-\sigma_{j-1})+\frac14(\sigma_0^3-\sigma_N^3)
\le\frac34\vert\lambda_1\vert(\sigma_N-\sigma_0)+\frac14(\sigma_0^3-\sigma_N^3)\,.
\label{eq:N1}
\end{align}
If $\sigma_0\ge0$ then also $\sigma_N\ge0$ and by Young's inequality
\begin{align}
\frac34\vert\lambda_1\vert\sigma_N\le\frac12\vert\lambda_1\vert^{\frac32}+\frac14\sigma_N^3\,.
\label{eq:Young}
\end{align}
If $\sigma_0<0$ then the inequality still holds true. To this end we note that by the min--max principle $\vert \lambda_1\vert \ge\sigma_0^2$ since $V\le0$ and since the operator without potential has a single negative eigenvalue $-\sigma_0^2$. Thus $(2\vert \lambda_1\vert ^{\frac12}+\sigma_N)\ge 0$ and from the identity
\begin{align}
\frac34\vert \lambda_1\vert \sigma_N
=\frac{1}{2}\vert \lambda_1\vert ^{\frac32}+\frac14\sigma_N^3-\frac{1}{4}(\vert \lambda_1\vert ^\frac12-\sigma_N)^2(2\vert \lambda_1\vert ^{\frac12}+\sigma_N)
\label{eq:Youngid} 
\end{align}
we again obtain \eqref{eq:Young}. Inserting \eqref{eq:Young} into \eqref{eq:N1} establishes that the inequality in Theorem \ref{th:main} implies the inequality in Theorem \ref{th:LTELU} if $V\in L^1(\R_+)\cap L^2(\R_+), V\le0$. The assumptions in the latter can then be relaxed to $V\in L^2(\R_+), V\le 0$ by the standard approximation arguments. 

We will provide an explicit example where the former inequality becomes an equality, while the latter remains a strict inequality.  To this end, we apply the double commutation method to insert a single eigenvalue into the spectrum of the free Schr{\"o}dinger operator $-\DD{x}$ with Neumann boundary condition $\varphi'(0)=0$. For fixed $\omega\in\R$ we consider $\varphi(x)=\cosh(\omega x)$, which satisfies $-\varphi''=-\omega^2\varphi$ as well as $\varphi'(0)=0$. Note that in contrast to the assumptions in Theorem \ref{th:double}, the function $\varphi$ is not an element of $L^2(\R_+)$. Furthermore we choose $\gamma>0$. From \cite[Theorem 3.2]{Gesztesy1996} we can conclude that the operator $-\DD{x}+V$ with potential
\begin{align*}
V(x)=-2\frac{\mathrm{d}}{\mathrm{d}x}\left(\frac{\gamma\cosh^2(\omega x)}{1+\gamma\int_0^x\cosh^2(\omega t)\dd t}\right)
\end{align*}
and Robin boundary condition $\varphi'(0)+\gamma\varphi(0)=0$ has a single negative eigenvalue $-\omega^2$. By construction (or by direct computation) the inequality of Theorem \ref{th:main} is found to be an equality in this case. In particular
\begin{align*}
\frac{3}{16}\int_0^\infty V(x)^2\dd x=\frac{1}{4}\gamma^3-\frac34\gamma\omega^2+\omega^3\,.
\end{align*}
The inequality of Theorem \ref{th:LTELU} on the other hand reduces to $\frac{\omega^3}{2}\le\omega^3$, which shows that for this example, the factor of $\frac12$ in front of the lowest eigenvalue is not necessary. 

Both inequalities are sharp for the free operator $-\DD{x}$ with boundary condition $\varphi'(0)-\sigma_0\varphi(0)=0$, which for $\sigma_0<0$ has a single negative eigenvalue $-\sigma_0^2$ with normalised eigenfunction $\varphi_1(x)=\sqrt{-2\sigma_0}\e^{\sigma_0 x}$. Under the assumptions of Theorem \ref{th:main}, the inequality of Theorem \ref{th:LTELU} cannot be an identity for potentials $V\in L^1(\R_+)\cap L^2(\R_+)$ with more than one eigenvalue, since the bound was proved by applying \eqref{eq:LTR} to the Dirichlet problem obtained after the initial step of the single commutation method. By Theorem \ref{th:Dir} this yields a strict inequality.  

\subsection{Some simplifications in special cases}
In some cases the bound in Theorem \ref{th:main} can be simplified such that it does not depend on the (often unknown) quantities $\sigma_j$ for $j\ge1$.

If $\sigma_0\ge0$ then Young's inequality allows us to conclude that
\begin{align*}
\frac34\vert\lambda_1\vert(\sigma_N-\sigma_0)+\frac14(\sigma_0^3-\sigma_N^3)
&\le\frac12\vert\lambda_1\vert^{\frac32}+\frac14(\sigma_N-\sigma_0)^3+\frac14(\sigma_0^3-\sigma_N^3)\\
&=\frac12\vert\lambda_1\vert^{\frac32}-\frac34\sigma_0\sigma_N(\sigma_N-\sigma_0)
\le\frac12\vert\lambda_1\vert^{\frac32}\,.
\end{align*}
From \eqref{eq:N1} we thus obtain that Theorem \ref{th:main} implies 
\begin{align*}
\frac12\vert\lambda_1\vert^{\frac32}+\sum_{j\ge2} \vert \lambda_j\vert ^\frac32
\le \frac{3}{16}\int_0^\infty V(x)^2\dd x\,. 
\end{align*}
While this result cannot be read off directly from the bound in Theorem~\ref{th:LTELU}, we note that it can be alternatively obtained by first applying the min-max principle and subsequently using Theorem \ref{th:LTELU} in the special case of a Neumann boundary condition~$\sigma_0=0$. 

More can be said if one can establish that $\sigma_0\ge\vert\lambda_1\vert^{1/2}$. In this case $(\vert\lambda_1\vert^{1/2}-\sigma_N)^2\ge(\vert\lambda_1\vert^{1/2}-\sigma_0)^2$ and thus \eqref{eq:Youngid} shows
\begin{align*}
\frac34\vert \lambda_1\vert \sigma_N
&\le\frac{1}{2}\vert \lambda_1\vert ^{\frac32}+\frac14\sigma_N^3-\frac{1}{4}(\vert \lambda_1\vert ^\frac12-\sigma_0)^2(2\vert \lambda_1\vert ^{\frac12}+\sigma_0)\\
&=\frac34\vert\lambda_1\vert\sigma_0-\frac14\sigma_0^3+\frac14\sigma_N^3\,.
\end{align*}
As a consequence 
\begin{align*}
\frac34\vert\lambda_1\vert(\sigma_N-\sigma_0)+\frac14(\sigma_0^3-\sigma_N^3)\le0
\end{align*}
and thus, on account of \eqref{eq:N1}, we obtain
\begin{align*}
\sum_{j\ge1} \vert \lambda_j\vert ^\frac32
\le \frac{3}{16}\int_0^\infty V(x)^2\dd x
\end{align*}
from Theorem \ref{th:main}. We observe that, in this special case, the Lieb--Thirring bound holds without any additional terms. It is not possible to obtain this result from Theorem \ref{th:LTELU} as the additional term in the inequality has the opposite sign, i.e.
\begin{align*}
\frac12\vert\lambda_1\vert^{\frac32}-\frac34\vert\lambda_1\vert\sigma_0+\frac14\sigma_0^3\ge0
\end{align*}
by Young's inequality. 

Lastly, if $\sigma_0\le 0$ and $V\le0$ then Young's inequality implies
\begin{align*}
-\frac34\vert\lambda_1\vert^{\frac12}\sigma_0\le\frac12\vert\lambda_1\vert^{\frac32}-\frac14\sigma_0^3
\end{align*}
and together with \eqref{eq:Young} and \eqref{eq:N1} we conclude that Theorem \ref{th:main} implies
\begin{align*}
\sum_{j\ge2} \vert \lambda_j\vert ^\frac32
\le \frac{3}{16}\int_0^\infty V(x)^2\dd x\,. 
\end{align*}
This result also follows from Theorem \ref{th:LTELU} by the same argument.

\subsection*{Acknowledgements}
The author was partially supported by VILLUM FONDEN through the QMATH Centre of Excellence (grant no.~10059) and  by VR grant 2017-04736 at the Royal Swedish Academy of Sciences. The author is thankful to the anonymous referee of an earlier version of this manuscript for their useful comments and to Ari Laptev for stimulating discussions on the topic of commutation methods.

\end{document}